\documentclass[12pt,reqno]{amsart}
\usepackage{cite}
\usepackage{amssymb,latexsym,amsmath,amsfonts}
\usepackage{latexsym}
\usepackage{amsmath, amsthm, amscd, amsfonts, amssymb,  color}
\usepackage[mathscr]{eucal}
= 8.9in
\theoremstyle{plain}
\newtheorem{theorem}{\textbf{Theorem}}[section]
\newtheorem{corollary}[theorem]{Corollary}

\newtheorem{remark}[theorem]{\textbf{Remark}}

\newcommand{\R}{\mathbb{R}}
\newcommand{\N}{\mathbb{N}}

\newtheorem{definition}[theorem]{Definition}

\newtheorem{example}[theorem]{\textbf{Example}}

\numberwithin{equation}{section}
\begin{document}

\title[A study on Kannan Type Contractive Mappings]
{ A study on Kannan Type Contractive Mappings}
\author[ H. Garai, T. Senapati, L.K. Dey]%
{ Hiranmoy Garai$^{1}$, Tanusri Senapati$^{2}$, Lakshmi Kanta Dey$^{3}$}

%\thanks{}
\address{{$^{1}$\,} Hiranmoy Garai,
                    Department of Mathematics,
                    National Institute of Technology
                    Durgapur,
                    West Bengal,
                    India.}
                    \email{hiran.garai24@gmail.com}

\address{{$^{2}$\,} Tanusri Senapati,
                    Department of Mathematics,
                    National Institute of Technology
                    Durgapur,
                    West Bengal,
                    India.}
                    \email{senapati.tanusri@gmail.com}
\address{{$^{3}$\,} Lakshmi Kanta Dey,
                    Department of Mathematics,
                    National Institute of Technology Durgapur,
                    West Bengal,
                    India.}
                    \email{lakshmikdey@yahoo.co.in}

%\address{{$^{3}$\,} Ankush Chanda,
%                    Department of Mathematics,
%                    National Institute of Technology
%                    Durgapur,
%                    West Bengal,
%                    India.}
%                    \email{ankushchanda8@gmail.com}

\thanks{*Corresponding author: lakshmikdey@yahoo.co.in (L.K. Dey)}

 \subjclass[2010]{ $47$H$10$, $54$H$25$. }

\begin{abstract}
In this article, we consider  Kannan type contractive self-map $T$ on a metric space $(X,d)$ such that \[d(Tx,Ty)<\frac{1}{2}\{d(x,Tx)+d(y,Ty)\} \mbox{ for all } x \neq y \in X,  \] and  establish some new fixed point results without taking the compactness  of $X$ and also without assuming continuity of $T$. Further,  we anticipate a result ensuring the completeness of the space $X$ via FPP of this map. Finally, we are able to give an affirmative answer to the open question posed by J. G\'{o}rnicki [\textit{Fixed point theorems for Kannan type mappings}, J. Fixed Point Theory Appl. 2017]. Apart from these, our manuscript consists of several non-trivial examples which signify the motivation of our investigations.
\end{abstract}
\maketitle

\section{\bf {Introduction and Preliminaries}}
It is almost a century where several mathematicians have improved, extended and enriched the classical Banach contraction principle\cite{banach} in different directions along with variety of applications. It is well-known that every Banach contractive mapping is a continuous function. In this sequel, it was a natural question does there exist any contractive map accompanied with fixed point which is not necessarily continuous? In 1968, R. Kannan \cite{kan} was the first mathematician who found the answer and presented the following fixed point result. 
\begin{theorem}\cite{kan}
Let $(X,d)$ be a complete metric space and $T$ be a self-mapping on $X$ satisfying  $$d(Tx,Ty)\leq k\{d(x,Tx)+d(y,Ty)\}$$ for all $x,y \in X$ and $k\in [0,\frac{1}{2})$.
Then $T$ has a unique fixed point $z\in X$, and for any $x\in X$ the sequence of iterates $(T^nx)$ converges to $z$.
\end{theorem}
Another beauty of the above result is that we can characterize the completeness of the underlying space $X$ in terms of fixed point of $T$. 
%Kannan fixed point theorem plays an important role in fixed point theory due to its own simplicity and beauty. This is the only result which gives the guarantee of completeness of the underlying spaces in terms of the existence of fixed point.
 In 1975, Subrahmanyam \cite{subra} proved that a metric space $(X,d)$ will be complete if and only if every Kannan mapping has a unique fixed point in $X$. Later on, B. Fisher \cite{fish} and M. S. Khan \cite{khan} simultaneously proved two fixed point results related to contractive type mappings. They proved that a continuous mapping on a compact metric space $(X,d)$ has a unique fixed point if $T$ satisfies  $$d(Tx,Ty)<\frac{1}{2}\{d(x,Ty)+d(y,Tx)\}$$ or $$d(Tx,Ty)<(d(x,Tx)d(y,Ty))^\frac{1}{2}$$ for all $x,y \in X$ with $x\neq y$ respectively.

 Soon there after, in 1980,  Chen and Yeh \cite{chen} extended the above two results in a more general way. In their article, they proved the following result as a corollary.
 \begin{theorem}\cite{chen}\label{chen}
Let $T$ be a continuous mapping of a non-empty compact metric space $(X,d)$ satisfying 
\begin{eqnarray*}
d(Tx,Ty)&<&\max \{d(x,y),\frac{1}{2}(d(x,Tx)+d(y,Ty)),\frac{1}{2}(d(x,Ty)+d(y,Tx)),\\
&&(d(x,y))^{-1}d(x,Tx)d(y,Ty),(d(x,Tx)d(y,Ty))^{\frac{1}{2}},\\
&&a(x,y)d(x,Ty)d(y,Tx),b(x,y)((d(x,Ty)d(y,Tx))^{\frac{1}{2}}\}
\end{eqnarray*}
for all $x,y\in X$ with $x\neq y$ and $a(x,y)$, $b(x,y)$ are two non-negative real functions, then $T$ has a fixed point. If, in addition, $a(x,y)\leq  (d(x,y))^{-1} $ and $b(x,y)\leq 1$, then $T$ has a unique fixed point.
\end{theorem}
Very recently, J. G\'{o}rnicki \cite{gorn}  proved the following result.
\begin{theorem}\cite{gorn}\label{gorn}
Let $(X,d)$ be a compact metric space and $T:X\rightarrow X$ be a continuous mapping satisfying
$$d(Tx,Ty)<\frac{1}{2}\{d(x,Tx)+d(y,Ty)\}$$
for all $x,y\in X$ with $x\neq y$. Then $T$ has a unique fixed point and for every $x\in X$, the sequence $(T^nx)$ converges to the fixed point. 
\end{theorem}
Unfortunately, the  above result does not contribute anything new in the literature, since one can easily get this result as a special case  of Theorem \ref{chen}, which was already proved long back, in $1980$.

Up till now, we observe that to ensure the existence of fixed points of Kannan type contractive mappings, both of continuity and compactness take crucial parts in the literature. In this article, our main aim is to investigate the existence of fixed point of Kannan type contractive mappings without assuming continuity  of the  mapping as well as the compactness property of the underlying spaces. In this direction, we successfully present some fixed point results which will be presented in next section. Before proceed further, we  recall some definitions which will be useful in our main results.  
\begin{definition}\cite{edw}
A metric space $(X,d)$ is said to be boundedly compact if every bounded sequence in $X$ has a convergent subsequence.
\end{definition}
It is clear from definition that every compact metric space is boundedly compact, but a boundedly compact metric space need not be compact, for example, the set of real numbers $\mathbb{R}$  with usual metric  is not compact but boundedy compact.
\begin{definition}
Let $(X,d)$ be a metric space and $T$ be a self mapping on $X$. Then orbit of $T$ at $x\in X$ is defined as $$O_x(T)=\{x,Tx,T^2x, T^3x,\dots\}.$$ 
\end{definition}
Now, we define the concept of $T$-orbitally compact set.
\begin{definition}
Let $(X,d)$ be a metric space and $T$ be a self-mapping on $X$, then $X$ is said to be $T$-orbitally compact if every sequence in $O_x(T)$  has a convergent subsequence for all $x\in X$. 
\end{definition}
 It is clear that $T$-orbitally compactness of a space depends on the mapping $T$ defined on it. 
\begin{example}
 Let $X=[0,\infty)$ be a metric space with respect to usual metric on $\R$. Define two mappings $T_1,T_2$ on $X$ by $$T_1 x= \frac{x}{n+1}, \mbox{ if } n-1\leq x <n,$$ and $$T_2x=2x $$  for all $ x\in X $ and $n\in \N$. Then clearly $X$ is $T_1$-orbitally compact but not $T_2$-orbitally compact.

\end{example}

 Moreover, it is easy to see that every compact metric space is $T$-orbitally compact but the converse is not true. Also note that boundedly compactness and  $T$-orbitally compactness are totally independent. Even $T$-orbitally compactness of $X$ does not give the guaranty to be complete. To show this, we consider the following examples.
 \begin{example}
Let $X=[0,1)$ endowed with the usual metric. Define $T:X\rightarrow X$  by $Tx=\frac{x}{2}$. Then it is easy to see that $X$ is $T$-orbitally compact but it is not complete.
\end{example}
\begin{example}
Let $(X,d)$ be a usual metric space with $X=[0,\infty)$. We define $T:X\rightarrow X$ by $Tx=2x$. Then it is trivial to check that $X$ is boundedly compact but not $T$-orbitally compact .
\end{example}

\section{\bf Main Results} 
Now, we are in a position to state our main results.
\begin{theorem}\label{thm1}
Let $(X,d)$ be a boundedly compact metric space and $T:X\rightarrow X$ be a mapping such that $$d(Tx,Ty)<\frac{1}{2}\{d(x,Tx)+d(y,Ty)\}$$ for all $x,y\in X$ with $x\neq y$. Then $T$ will be a Picard operator.
\end{theorem}
\begin{proof}
Let $x_0\in X$ be arbitrary but fixed  and consider the iterated sequence $(x_n)$ where $x_n=T^nx_0$ for each $n\in \mathbb{N}$. If the sequence $(x_n)$ has two equal consecutive terms, then $T$ must have fixed point. So, we consider that no two consecutive terms of $(x_n)$ are equal. We denote $s_n=d(x_n,x_{n+1})$ for each $n\in \mathbb{N}$. Then we have
\begin{eqnarray*}
s_n & = & d(T^nx_0,T^{n+1}x_0)\\
& = & d(T(T^{n-1}x_0),T(T^nx_0)\\
& < & \frac{1}{2}\{d(T^{n-1}x_0,T^nx_0)+d(T^nx_0,T^{n+1}x_0)\}\\
& =& \frac{1}{2}(s_{n-1}+s_n)\\
 \Rightarrow s_n & < & s_{n-1}.
\end{eqnarray*}
This shows that $(s_n)$ is a strictly decreasing sequence of real numbers and also the sequence is bounded below, so it must be a convergent sequence. For each $n\in \mathbb{N}$, we must have, \[s_n<s_{n-1} < \dots < s_1=K(say).\]
 
Again, for all $n,m \in \mathbb{N}$, we deduce, $$d(x_n,x_m)<\frac{1}{2}(s_{n-1}+s_{m-1})<K.$$
Therefore, $(x_n)$ is a bounded sequence in $X$.  By boundedly compactness property of $X$, $(x_n)$ must have a convergent subsequence, say $(x_{n_k})$ which converges to some $z \in X$.

Therefore, $$\displaystyle \lim_{k\to \infty} d(x_{n_k},x_{n_k+1})=d(\displaystyle \lim_{k\to \infty}x_{n_k},\displaystyle \lim_{k\to \infty} x_{n_k+1})=d(z,z)=0.$$
This shows that the convergent sequence $(s_n)$ contains a subsequence $(s_{n_k})$ which itself converges to 0. So the sequence $(s_n)$  must converge to 0. Hence, for all $n,m \in \mathbb{N}$, $$d(x_n,x_m)<\frac{1}{2}\{d(T^{n-1}x_0),T^nx_0)+d(T^{m-1}x_0,T^mx_0)\}\to 0$$ as $n,m\to \infty$.

This deduces that $(x_n)$ is a Cauchy sequence. As the subsequence $(x_{n_k})$ of $(x_n)$ converges to $z$, so the limit of $(x_n)$ must be $z$. Also, we have
\begin{eqnarray*}
d(z,Tz)& \leq & d(z,T^{n+1}x_0)+d(T^{n+1}x_0, Tz)\\
& < & d(z,T^{n+1}x_0) +  \frac{1}{2}\{d(T^nx_0,T^{n+1}x_0)+d(z,Tz)\}\\
\Rightarrow \frac{1}{2}d(z,Tz) & < &d(z,T^{n+1}x_0)+\frac{1}{2}d(T^nx_0,T^{n+1}x_0)\to 0 \mbox { as n}\to\infty.
\end{eqnarray*}
This implies that  $z=Tz$, i.e., $z$ is a fixed point of $T$.

 Next, we check the uniqueness of $z$. Arguing by contradiction, let $z^*$ be another fixed point of $T$, then
\begin{eqnarray*}
d(z,z^*)&=&d(Tz,Tz^*)\\
& < & \frac{1}{2}\{d(z,Tz)+d(z^*,Tz^*)\}\\
\Rightarrow d(z,z^*)&<&0,
\end{eqnarray*} which leads us to a contradiction. Hence, our assumption was wrong. Therefore, $z$ must be the unique fixed point $T$. Since, we take $x_0$ as an arbitrary point, so for every $x\in X$, the iterated sequence $(T^nx)$ converges to $z$,  i.e., $T$ is a Picard operator.
\end{proof}
One can deduce the following result as an immediate consequence of the above result. 
\begin{corollary}
Let $(X,d)$ be a boundedly compact metric space and $T:X\rightarrow X$ be a mapping such that $$d(T^{m+1}x,T^{m+1}y)<\frac{1}{2}\{d(T^mx,T^{m+1}x)+d(T^my,T^{m+1}y)\}$$ for all $x,y \in X$ with $x\neq y$ for some positive integer $m\in N$. Then $T$ has  unique fixed point $z$ and for any $x \in X$ the sequence of iterates $(T^nx)$ converges to $z$.
\end{corollary}
 We have already shown that boundedly compactness and $T$-orbitally compactness are independent. In our next result, we show the existence of fixed point of $T$ by taking $T$-orbitally compactness instead of boundedly compactness. 

\begin{theorem}
Let $(X,d)$ be a $T$-orbitally compact metric space where $T:X\rightarrow X$ is a mapping such that $$d(Tx,Ty)<\frac{1}{2}\{d(x,Tx)+d(y,Ty)\}$$ for all $x,y\in X$ with $x\neq y$, then $T$ has a unique fixed point $z$ and for any $x \in X$ the sequence of iterates $(T^nx)$ converges to $z$.
\end{theorem}
\begin{proof}
Let $x_0\in X$ be arbitrary but fixed and consider the sequence $(x_n)$ where $x_n=T^nx_0$ for each $n \in \mathbb{N}$. Since $X$ is $T$-orbitally compact, so the sequence $(x_n)$ has a convergent subsequence, say $(x_{n_k})$ and let $(x_{n_k})$ converge to $z$ in $X$.

Now, $$\displaystyle \lim_{k\to \infty} d(x_{n_k},x_{n_k+1})=d(\displaystyle \lim_{k\to \infty}x_{n_k},\displaystyle \lim_{k\to \infty} x_{n_k+1})=d(z,z)=0.$$
Thus we see that the convergent sequence $(d(x_n,x_{n+1}))$ contains the subsequence  $(d(x_{n_k},x_{n_k+1}))$ which converges to $0$, so the sequence $(d(x_n,x_{n+1}))$ itself converges to $0$.

For all $n,m \in \mathbb{N}$, we have, $$d(x_n,x_m)<\frac{1}{2}(s_{n-1}+s_{m-1}) \to 0$$ as $n,m \to\infty$. This implies that the sequence $(x_n)$ is a Cauchy sequence and  $x_n \to z\in X$ as $n\to \infty$. Next, we claim that $z$ is a fixed point of $T$.
\begin{eqnarray*}
d(z,Tz)& \leq & d(z,T^{n+1}x_0)+d(T^{n+1}x_0, Tz)\\
& < & d(z,T^{n+1}x_0) +  \frac{1}{2}\{d(T^nx_0,T^{n+1}x_0+d(z,Tz)\}\\
\Rightarrow \frac{1}{2}d(z,Tz) & < &d(z,T^{n+1}x_0)+\frac{1}{2}d(T^nx_0,T^{n+1}x_0)\to 0 \mbox { as n}\to\infty.
\end{eqnarray*}
Therefore $z=Tz$ and thus $z$ is a fixed point $T$. Also it is easy to check that $z$ is the only fixed point of $X$.
\end{proof}
In support of our result, we present the following example. It is  surprise to note that even a discontinuous  Kannan type contractive self-map $T$ on an incomplete  metric space $X$  has a unique fixed point.
\begin{example}
Let $X=(1,2] \cup \{-1,0\}$ and define $T:X \rightarrow X$ by
 \[ T(x) = \left\{ \begin{array}{ll}
-1 & {x=2};\\
 0 & {x\neq 2 }.\end{array} \right. \]
Now, for $x\neq 2$, we have
$$d(Tx,T2)=|Tx-T2|=|0+1|=1$$ whereas $$\frac{1}{2}\{d(x,Tx)+d(2,T2)\}=\frac{1}{2}\{|x-Tx|+|2+1|\}>1.$$
So, $$d(Tx,T2)<\frac{1}{2}\{d(x,Tx)+d(2,T2)\}.$$
Again, for $x,y \in X$ with $x,y \neq 2$ and $x \neq y$, we have $d(Tx,Ty)=0$ but 
$$\frac{1}{2}\{d(x,Tx)+d(y,Ty)\}=\frac{1}{2}\{|x-Tx|+|y-Ty|\}>0.$$
Therefore, we have $d(Tx,Ty)<\frac{1}{2}\{d(x,Tx)+d(y,Ty)\}.$

Clearly, $X$ is non-compete $T$-orbitally compact and $T$ is a discontinuous mapping satisfying  $d(Tx,Ty)<\frac{1}{2}\{d(x,Tx)+d(y,Ty)\}$ for all $x,y\in X$ with $x\neq y$. But still $T$ has a fixed point and $0$ is the only fixed point $T$. 
\end{example}
Next, we  present an important result which provides a sufficient condition of completeness of the underlying space via fixed point property (FPP).
\begin{theorem}
Let every self-mapping $T$ on a metric space $(X,d)$ satisfying $$d(Tx,Ty)<\frac{1}{2}\{d(x,Tx)+d(y,Ty)\}$$ for all $x,y \in X$ with $x\neq y,$ has a unique fixed point. Then $(X,d)$ must be a complete metric space. 
\end{theorem}
\begin{proof}
Arguing by contradiction, assume that $(X,d)$ is not complete, then there must be a Cauchy sequence $(x_n)$ in $X$ which is not convergent in $X$. Without loss of generality, we assume that all terms of the sequence $(x_n)$ are distinct. Now, we consider the following set
 $$A=\{x_n:n\in \mathbb{N}\}.$$
 As the sequence $(x_n)$ does not converge in $X$, so $d(x,A)>0$ for all $x\in X \backslash A$.
 
Let $x\in X$ be any arbitrary point. If $x\in X\backslash A$, then we can find an integer $n_x\in \mathbb{N}$ such that 
\begin{eqnarray}
d(x_m,x_{n_x})&<&  \frac{1}{2}d(x,A), ~\forall m\geq n_x\nonumber\\
&\leq &\frac{1}{2} d(x,x_n),~\forall n \in \mathbb{N}\nonumber\\
\Rightarrow d(x_m,x_{n_x})&<& \frac{1}{2}d(x,x_n),~ \forall m\geq n_x ~\mbox{and}~ n\in \mathbb{N}.\label{equ1}
\end{eqnarray}

Let $x\in A$, then $x=x_{n_0}$, for some $n_0\in \mathbb{N}$. Again, we can find some $n_0'\in \mathbb{N}$ such that 
\begin{equation}\label{equ2}
(x_m,x_{n_0'})<\frac{1}{2}d(x_n,x_{n_0})
\end{equation} for all $n\in \mathbb{N}$ and $m\geq n_0'>n_0.$

Next, we define $T:X\rightarrow X$ by
$$Tx=
\begin{cases}
x_{n_x}, \text{if x }\in X \backslash A;\\
x_{n_0'}, \text{if x}\in A \mbox{ and } x=x_{n_0}.
\end{cases}$$

Let $x,y \in X$ be arbitrary points with $x\neq y$. If $x,y \in X\backslash A$, then $Tx=x_{n_x}$ and $Ty=y_{n_y}$. Without loss of generality, we assume that $n_y\geq n_x$. Then from Equation \ref{equ1}, we have  $$d(y_{n_y},x_{n_x})<\frac{1}{2}(x,x_{n_x})=\frac{1}{2}d(x,Tx)$$ which yields $$d(Tx,Ty)<\frac{1}{2}\{d(x,Tx)+d(y,Ty)\}.$$
Next, we consider that $x,y \in A$, and $x=x_{n_0}$; $y=y_{m_0}$ for some $n_0,m_0\in \mathbb{N}$. Then $Tx=x_{n_0'}$ and $Ty=y_{m_0'}$. Without loss of generality,  we assume that $m_0'\geq n_0'$. Then from Equation \ref{equ2}, we can deduce that
 $$T(x_{m_0'},x_{n_0'})< \frac{1}{2}d(x_{n_0'},x_{n_0})=\frac{1}{2}d(x,Tx)$$ which gives $$d(Tx,Ty)<\frac{1}{2}\{d(x,Tx)+d(y,Ty)\}.$$
 Finally, we consider another possibility such that $x\in X\backslash A$ and $y\in A$.  We set $y=x_{n_0}$ for some $n_0\in \mathbb{N}$.  Then, we obtain $Tx=x_{n_x}$ and $Ty=x_{n_0'}$. If $n_0'\geq n_x$, then from Equation \ref{equ1}, we get, $$d(x_{n_0'}, x_{n_x})< \frac{1}{2}d(x,x_{n_x})=\frac{1}{2}d(x,Tx)$$ which implies $$d(Tx,Ty)<\frac{1}{2}\{d(x,Tx)+d(y,Ty)\}.$$
 
For $n_x>x_{n_0'}$, Equation \ref{equ2} yields $$d(x_{n_x},x_{n_0'})<\frac{1}{2}d(x_{n_0'},x_{n_0})=\frac{1}{2}d(y,Ty)$$ which deduces $$d(Tx,Ty)<\frac{1}{2}\{d(x,Tx)+d(y,Ty)\}.$$ 
Therefore, for all $x,y \in X$ with $x\neq y$, we have
 $$d(Tx,Ty)<\frac{1}{2}\{d(x,Tx)+d(y,Ty)\},$$ i.e., $T$ is a Kannan type contractive map which has no fixed point. This leads us to a contradiction. Hence, our assumption was wrong. Therefore, $(X,d)$ must be a complete metric space.
\end{proof}

Recently, J. G\'{o}rinicki \cite{gorn} raised the following open question:

\noindent{\textbf{Question:}} Does there exists a complete but noncompact metric space $(X,d)$ and a continuous mapping $T:X\rightarrow X$ such that $$d(Tx,Ty)<\frac{1}{2}\{d(x,Tx)+d(y,Ty)\}$$ for all $x,y\in X$ with $x\neq y$ and $T$ is fixed point free?

We  give an affirmative answer of the above question which is presented below.
\begin{example}
Let us choose  $X=\mathbb{N}$ and we define $d:X\times X\rightarrow \mathbb{R}$ by
\[    d(x,y) = \left\{\begin{array}{ll}
        1+|\frac{1}{x}-\frac{1}{y}|, & \text{if } x\neq y\\
       0, & \text{if } x=y.
        \end{array}\right.  \]
%$$
%d(x,y)=
%\begin{cases} 
%1+|\frac{1}{x}-\frac{1}{y}|, \text{ if x\neq y };\\
%0, \text{ if x=y }.
%\end{cases}
%$$
Clearly $d$ is a metric on $X$. Note that every Cauchy sequence in $X$  is eventually constant and hence  $(X,d)$ is a complete metric space. However, this is non-compact as the sequence $(n)$ has no convergent subsequence in $X$. Now, we define a function $T:X\rightarrow X$ by $$Tx=3x$$ for all $x\in X$. It is easy to verify that $T$ is continuous  and a fixed point free map. Now, it is remaining to show that $T$ satisfies the Kannan type contractive condition. In order to do this, we choose $x,y\in X$ with $x<y$, then,  
\begin{eqnarray*}
d(Tx,Ty)&=& 1+|\frac{1}{3x}-\frac{1}{3y}|\\
&=& 1+\frac{1}{3x}-\frac{1}{3y}< 1+\frac{1}{3x};
\end{eqnarray*}
whereas, 
\begin{eqnarray*}
\frac{1}{2}\{d(x,Tx)+d(y,Ty)\}&=&\frac{1}{2}\big\{1+|\frac{1}{x}-\frac{1}{3x}|+1+|\frac{1}{y}-\frac{1}{3y}|\big\}\\
&=&1+\frac{1}{3x}+\frac{1}{3y}>1+\frac{1}{3x}.
\end{eqnarray*}
Therefore, $$d(Tx,Ty)<\frac{1}{2}\{d(x,Tx)+d(y,Ty)\}$$ for all $x,y$ in $X$ with $x<y$.

Similarly, one can prove it for the case $x,y\in X$ with $x>y$.
Now from the above arguments, we can say that  
$$d(Tx,Ty)<\frac{1}{2}\{d(x,Tx)+d(y,Ty)\}$$ for all $x,y$ in $X$ with $x\neq y$. Thus, $T$ is a Kannan type contractive map but it is fixed point free.
\end{example}
%The above example shows that there may exist \textbf{\textit{a continuous Kannan type contractive mapping on a complete but noncompact metric space which is not fixed point free}}. 
\begin{remark}
There are special type of complete but non-compact spaces, where, we can always  find fixed point of this type map. For example, if we take $X$ as  a closed subset of $\mathbb{R}^n$, ($n\geq 1$), equipped with the usual metric, then every Kannan type contactive self-mapping on $X$ must have a unique fixed point. This is clear from   Theorem \ref{thm1}.
\end{remark}

However, the following result gives the guaranty of existence of unique fixed point of Kannan type contactive self-map $T$ on an arbitrary complete metric space $X$ with a mild additional condition  on $T$.
\begin{theorem}
Let $(X,d)$ be a complete metric space and $T$ be a self-mapping on $X$ such that
\begin{enumerate}
\item[i)]$d(Tx,Ty)<\frac{1}{2}\{d(x,Tx)+d(y,Ty)\}$ for all $x,y$ in $X$ with $x\neq y$,
\item[ii)] For any $x\in X$ and for any $\epsilon>0$, there exists $\delta>0$ such that $d(T^ix,T^jx)<\epsilon + \delta $ implies $d(T^{i+1}x,T^{j+1}x)\leq \epsilon$ for any $i,j \in \mathbb{N} \cup \{0\}$.
\end{enumerate}
Then $T$ has a unique fixed point $z$ and for any $x \in X$ the sequence of iterates $T^nx$ converges to $z$.
\end{theorem}
\begin{proof}
Let $x_0\in X$ be arbitrary but fixed and consider the sequence $(x_n)$ where $x_n=T^nx_0$ for each $n\in \mathbb{N}$. Let the sequence $(x_n)$ do not have two equal consecutive terms, i.e., $x_n\neq x_{n+1}$ for all $n\in \mathbb{N}$. Then it is not difficult to check that the sequence of real numbers $(s_n)$ where $s_n=d(x_n,x_{n+1})$ is a decreasing sequence and also this sequence is bounded below, so this sequence is convergent and let $\displaystyle \lim_{n\to \infty} s_n=b$, so $b\geq 0$.

Now if $b>0$, then by definition of $b$, there exists $\delta>0$ and $n\in \mathbb{N}$ such that 
\begin{eqnarray*}
&&s_n < b+ \delta\\
&\Rightarrow & d(x_n,x_{n+1})< b+ \delta.
\end{eqnarray*}
So by given condition, we have,  $d(x_{n+1},x_{n+2})\leq b$, i.e., $s_{n+1}\leq b$. This leads to a contradiction, so we must have b=0. Therefore, $ \displaystyle \lim_{n\to \infty} d(x_n,x_{n+1})=0$.

Now for any $n,m \in \mathbb{N}$, we have,$$d(x_n,x_m)<\frac{1}{2}\{d(T^{n-1}x_0,T^nx_0)+d(T^{m-1}x_0,T^mx_0)\}\to 0$$ as $n\to \infty$. This deduces that $(x_n)$ is a Cauchy sequence and hence convergent in $X$. Let us consider $$\displaystyle \lim_{n\to \infty} x_n=z.$$
Now, we show that $z$ is fixed point of $T$. In order to show this, we have
\begin{eqnarray*}
d(z,Tz)& \leq & d(z,T^{n+1}x_0)+d(T^{n+1}x_0, Tz)\\
& < & d(z,T^{n+1}x_0) +  \frac{1}{2}\{d(T^nx_0,T^{n+1}x_0)+d(z,Tz)\}\\
\Rightarrow \frac{1}{2}d(z,Tz) & < &d(z,T^{n+1}x_0)+\frac{1}{2}d(T^nx_0,T^{n+1}x_0)\to 0 \mbox { as n}\to\infty.
\end{eqnarray*}
This implies that $z=Tz$, i.e., $z$ is a fixed point of $T$. The uniqueness of the fixed point follows from Theorem \ref{thm1}.
\end{proof} 

\vskip.5cm\noindent{\bf Acknowledgements}\\
 The first named author would like to express his sincere gratitude to CSIR, New Delhi, India for their financial supports.

%\bibliography{bibfile}
%%\bibliographystyle{elsarticle-num}
\bibliographystyle{plain}

\end{document}